\documentclass[12pt,twoside]{amsart}
\usepackage{amsmath, amssymb,tabularx,supertabular,amsmath, verbatim, doc, amsthm, amssymb, mathrsfs, manfnt}

\evensidemargin \oddsidemargin
\newtheorem{thm}{Theorem}[section]
\newtheorem{theorem}[thm]{Theorem}

\newtheorem{proposition}[thm]{Proposition}

\theoremstyle{definition}

\theoremstyle{remark}

\title{Characterizing Hilbert modular cusp forms by coefficient size}

\author {Benjamin Linowitz}
\address{Department of Mathematics\\ 
6188 Kemeny Hall\\
Dartmouth College\\
Hanover, NH 03755}

\email[] {benjamin.linowitz@dartmouth.edu}
\urladdr{http://www.math.dartmouth.edu/~linowitz/ }

\newcommand{\scrM}{\mathscr{M}}
\newcommand{\scrE}{\mathscr{E}}
\newcommand{\scrS}{\mathscr{S}}

\newcommand{\Q}{\mathbb Q}
\newcommand{\C}{\mathbb C}
\newcommand{\Z}{\mathbb{Z}}
\newcommand{\T}{\mathbb{T}}

\newcommand{\calN}{\mathcal{N}}
\newcommand{\calM}{\mathcal{M}}

\newcommand{\calO}{\mathcal O}

\DeclareMathOperator*{\dete}{det}

\newcommand{\frakm}{{\mathfrak m}}
\newcommand{\frakd}{{\mathfrak d}}
\newcommand{\frakf}{{\mathfrak f}}
\newcommand{\frakp}{{\mathfrak p}}
\newcommand{\frakP}{{\mathfrak P}}
\newcommand{\frakq}{{\mathfrak q}}
\newcommand{\fraka}{{\mathfrak a}}
\newcommand{\frakb}{{\mathfrak b}}
\newcommand{\frakc}{{\mathfrak c}}

\parskip=\medskipamount

\begin{document}

\begin{abstract}
Associated to an (adelic) Hilbert modular form is a sequence of `Fourier coefficients' which uniquely determine the form. In this paper we characterize Hilbert modular cusp forms by the size of their Fourier coefficients. This answers in the affirmative a question posed by Winfried Kohnen.
\end{abstract}

\maketitle

\section{Introduction}\label{section:intro}

In \cite{Kohnen-OnCertainGMFs,Schmoll} it is shown that if the Fourier coefficients $a(n)$ ($n\geq 1$) of an elliptic modular form $f$ of even integral weight $k\geq 2$ and level $N$ satisfy the Deligne bound $a(n)\ll_{f,\epsilon} n^{(k-1)/2+\epsilon}$ ($\epsilon>0$), then $f$ must be a cusp form. An analogous result was recently proven by Kohnen and Martin \cite{Kohnen-Martin} for Siegel modular forms of weight $k$ and genus $2$ on the full Siegel modular group. Kohnen posed \cite{Kohnen-question} the question of whether one could prove analogous results in more general settings, in particular the Hilbert modular setting. This paper answers Kohnen's question in the affirmative by characterizing Hilbert modular cusp forms in terms of the size of their Fourier coefficients.

In seeking to generalize the results of \cite{Kohnen-OnCertainGMFs,Schmoll} to Hilbert modular forms one immediately confronts a number of difficulties: the lack of Fourier expansions and the absence of an action of Hecke operators under which the space of Hilbert modular forms is invariant. Let $K$ be a totally real number field and consider the complex vector space $M_k(\calN)$ of classical Hilbert modular forms of weight $k$ and level $\calN$ over $K$. If the strict class number $h^+$ of $K$ is greater than $1$, then a form $f\in M_k(\calN)$ need not possess a Fourier expansion and hence Fourier coefficients to examine. In order to circumvent this difficulty we work with the larger space $\scrM_k(\calN)$ of adelic Hilbert modular forms of weight $k$ and level $\calN$. The elements of $\scrM_k(\calN)$ are $h^+$-tuples of classical Hilbert modular forms and to each form $f\in \scrM_k(\calN)$ one associates (as in \cite{shimura-duke}) a sequence $\{ C(\frakm, f) : \frakm \mbox{ an integral ideal of } K\}$ of `Fourier coefficients' which uniquely determine $f$. Equally important to our analysis is the fact that unlike $M_k(\calN)$, the space of adelic Hilbert modular forms $\scrM_k(\calN)$ is mapped to itself under the action of the Hecke operators $T_\frakp$.

\begin{theorem}\label{theorem:intro} Let notation be as above with $k\geq 3$ and fix a real number $\epsilon\in \left(0, \frac{5k-7}{10}\right]$. If $f\in \scrM_k(\calN)$ satisfies $$|C(\frakm,f)|\ll_{f,\epsilon} N(\frakm)^{k-1-\epsilon }$$ then $f$ is a cusp form.
\end{theorem}

The proof of Theorem \ref{theorem:intro} proceeds along the same lines as those of \cite{Kohnen-OnCertainGMFs,Schmoll}, though is necessarily more complicated due to the technical nature of adelic Hilbert Eisenstein series. Crucial in the proof will be the structure of spaces of Hilbert Eisenstein series and in particular the newform theory of these spaces. This theory was proven in the elliptic case by Weisinger \cite{weisinger-thesis} and was later extended to the Hilbert modular setting by Wiles \cite{wiles} and Atwill and Linowitz \cite{Linowitz-Atwill}.

\section{Notation}\label{section:notation}

We employ the notation of \cite{shimura-duke, Linowitz-Atwill}. However, to make this paper somewhat self-contained, we shall briefly review the basic definitions of the functions and operators which we shall study.

Let $K$ be a totally real number field of degree $n$ over $\Q$ with ring of integers $\calO$, group of units $\calO^\times$, and totally positive units $\calO^\times_+$. Let $\frakd$ be the different of $K$. If $\frakq$ is a finite prime of $K$, we denote by $K_{\frakq}$ the completion of $K$ at $\frakq$, $\calO_{\frakq}$ the valuation ring of $K_{\frakq}$, and $\pi_{\frakq}$ a local uniformizer. Finally, we let $\mathfrak P_\infty$ denote the product of all archimedean primes of $K$. 

We denote by $K_\mathbb A$ the ring of $K$-adeles and by $K_\mathbb A^\times$ the group of $K$-ideles. As usual we view $K$ as a subgroup of $K_\mathbb A$ via the diagonal embedding. If $\tilde{\alpha}\in K_\mathbb A^\times$, we let $\tilde{\alpha}_{\infty}$ denote the archimedean part of $\tilde{\alpha}$ and $\tilde{\alpha}_0$ the finite part of $\tilde{\alpha}$. If $\mathcal{J}$ is an integral ideal we let $\tilde{\alpha}_{\mathcal{J}}$ denote the $\mathcal{J}$-part of $\tilde{\alpha}$. 

For an integral ideal $\calN$ we define a numerical character $\psi$ modulo $\calN$ to be a homomorphism $\psi: (\calO/\calN)^\times \rightarrow \C^\times$ and denote the conductor of $\psi$ by $\mathfrak{f}_\psi$. A Hecke character is a continuous character on the idele class group: $\Psi: K_\mathbb A^\times/ K^\times \rightarrow \C^\times$. We denote the induced character on $K_\mathbb A^\times$ by $\Psi$ as well. We adopt the convention that $\psi$ denotes a numerical character and $\Psi$ denotes a Hecke character.

For a fractional ideal $\mathcal{I}$ and integral ideal $\calN$, define 
  $$\Gamma_0(\calN,\mathcal{I})=\left\{ A\in \left(
  \begin{array}{ c c }
    \calO& \mathcal{I}^{-1}\frakd^{-1} \\
      \calN\mathcal{I}\frakd & \calO
  \end{array} \right) : \dete A \in \calO^\times_+ \right\}.$$

Let $\theta : \calO^\times_+\rightarrow \C^\times$ be a character of finite order and note that there exists an element $m\in\mathbb{R}^n$ such that $\theta(a)=a^{im}$ for all totally positive units $a$. While such an $m$ is not unique, we shall fix one such $m$ for the remainder of this paper. 

Let $k=(k_1,...,k_n)\in \mathbb{Z}_+^n$ and $\psi$ be a numerical character modulo $\calN$. Following Shimura  \cite{shimura-duke}, we define $M_k(\Gamma_0(\calN,\mathcal{I}),\psi,\theta)$ to be the complex vector space of classical Hilbert modular forms on $\Gamma_0(\calN,\mathcal I)$. 

It is well-known that the space of classical Hilbert modular forms of a fixed weight, character and congruence subgroup is not invariant under the algebra generated by the Hecke operators $T_\frakp$. We therefore consider the larger space of adelic Hilbert modular forms, which \textit{is} invariant under the Hecke algebra. Our construction follows that of Shimura \cite{shimura-duke}.

Fix a set of strict ideal class representatives $\mathcal{I}_1,...,\mathcal{I}_h$ of $K$, set $\Gamma_{\lambda}=\Gamma_0(\calN,\mathcal{I}_{\lambda})$, and put $$\scrM_k(\calN,\psi,\theta)=\prod_{\lambda=1}^h M_k(\Gamma_{\lambda},\psi,\theta).$$ In the case that both $\psi$ and $\theta$ are trivial characters we will simplify our notation and denote $\scrM_k(\calN,\psi,\theta)$ by $\scrM_k(\calN)$.

Let $G_\mathbb A=GL_2(K_\mathbb A)$ and view $G_K=GL_2(K)$ as a subgroup of $G_\mathbb A$ via the diagonal embedding. Denote by $G_{\infty} = GL_2(\mathbb{R})^n$ the archimedean part of $G_\mathbb A$. For an integral ideal $\calN$ of $\calO$ and a prime $\frakp$, let 
$$Y_{\frakp}(\calN)=\left\{ A=\left(
  \begin{array}{ c c }
    a& b \\
      c&d
  \end{array} \right) \in  \left(
  \begin{array}{ c c }
    \calO_{\frakp}&\frakd^{-1}\calO_{\frakp} \\
      \calN\frakd\calO_{\frakp} & \calO_{\frakp}
  \end{array} \right) : \dete A\in K_{\frakp}^\times, (a\calO_{\frakp},\calN\calO_{\frakp})=1  \right\},$$
  
  $$W_{\frakp}(\calN)=\{ A\in Y_{\frakp}(\calN) : \dete A\in \calO^\times_{\frakp}  \}$$
  and put $$Y=Y(\calN)=G_\mathbb A\cap \left(G_{\infty +}\times \prod_{\frakp\nmid\mathfrak P_\infty} Y_{\frakp}(\calN)\right), \qquad W=W(\calN)=G_{\infty +}\times \prod_{\frakp \nmid\mathfrak P_\infty} W_{\frakp}(\calN).$$

Given a numerical character $\psi$ modulo $\calN$ define a homomorphism $\psi_Y: Y\rightarrow \C^\times$ by setting $\psi_Y(\left(
  \begin{array}{ c c }
    \tilde{a}& * \\
      *&*
  \end{array} \right))=\psi(\tilde{a}_{\calN}\mbox{ mod }\calN )$.
  
Given a fractional ideal $\mathcal I$ of $K$ define $\tilde{\mathcal{I}}=(\mathcal{I}_{\nu})_{\nu}$ to be a fixed idele such that $\mathcal{I}_{\infty}=1$ and $\tilde{\mathcal{I}}\calO=\mathcal{I}$. For $\lambda=1,...,h,$ set $x_{\lambda}=\left(
  \begin{array}{ c c }
1& 0 \\
      0&\tilde{I}_{\lambda}
  \end{array} \right)\in G_\mathbb A$. By the Strong Approximation theorem $$G_\mathbb A=\bigcup_{\lambda=1}^h G_K x_{\lambda} W=\bigcup_{\lambda=1}^h G_K x_{\lambda}^{-\iota} W,$$ where $\iota$ denotes the canonical involution on two-by-two matrices.
  
  For an $h$-tuple $(f_1,...,f_h)\in\scrM_k(\calN,\psi,\theta)$ we define a function $f: G_\mathbb A\rightarrow \C$ by
  
  $$f(\alpha x_{\lambda}^{-\iota}w)=\psi_Y(w^{\iota})\dete(w_{\infty})^{im}(f_{\lambda}\mid w_{\infty})(\textbf{i})$$
for $\alpha\in G_K$, $w\in W(\calN)$ and $\textbf{i}=(i,...,i)$ (with $i=\sqrt{-1}$). Here $$f_{\lambda}\mid \left(
  \begin{array}{ c c }
a& b \\
      c&d
  \end{array} \right)(\tau)=(ad-bc)^{\frac{k}{2}}(c\tau+d)^{-k} f_{\lambda}\left(\frac{a\tau+b}{c\tau+d}\right).$$

We identify $\scrM_k(\calN,\psi,\theta)$ with the set of functions $f: G_\mathbb A\rightarrow \C$ satisfying
\begin{enumerate}
\item $f(\alpha x w)=\psi_Y(w^{\iota})f(x)$ for all $\alpha\in G_K, x\in G_\mathbb A, w\in W(\calN), w_{\infty}=1$
\item For each $\lambda$ there exists an element $f_{\lambda}\in M_k$ such that $$f(x_{\lambda}^{-\iota}y)=\dete(y)^{im}(f_{\lambda}\mid y)(\textbf{i})$$ for all $y\in G_{\infty +}$.
\end{enumerate}  
 
Let $\psi_{\infty}: K_\mathbb A^\times\rightarrow \C^\times$ be defined by $\psi_{\infty}(\tilde{a})=\mbox{sgn}(\tilde{a}_{\infty})^k|\tilde{a}_{\infty}|^{2im}$, where $m$ was specified in the definition of $\theta$. We say that a Hecke character $\Psi$ extends $\psi\psi_{\infty}$ if $\Psi(\tilde{a})=\psi(\tilde{a}_{\calN}\mbox{ mod }\calN)\psi_{\infty}(\tilde{a})$ for all $\tilde{a}\in K_{\infty}^\times \times \prod_{\frakp} \calO_{\frakp}^\times$. If the previous equality holds for $\psi_\infty(\tilde{a})=\mbox{sgn}(\tilde{a}_\infty)^k$ then we say that $\Psi$ extends $\psi$.

Given a Hecke character $\Psi$ extending $\psi\psi_{\infty}$ we define an ideal character $\Psi^*$ modulo $\calN\mathfrak{P}_{\infty}$ by 
\begin{displaymath}
\left\{ \begin{array}{ll}
\Psi^*(\frakp)=\Psi(\tilde{\pi}_{\frakp}) & \textrm{for $\frakp\nmid \calN$ and $\tilde{\pi}\calO=\frakp,$}\\
\Psi^*(\mathfrak{a})=0 & \textrm{if $(\mathfrak{a},\calN)\neq 1$ }\\
\end{array} \right.
\end{displaymath}

For $\tilde{s}\in K_\mathbb A^\times$, define $f^{\tilde{s}}(x)=f(\tilde{s}x)$. The map $\tilde{s}\longrightarrow \left( f\mapsto f^{\tilde{s}}\right)$ defines a unitary representation of $K_\mathbb A^\times$ in $\scrM_k(\calN,\psi,\theta)$. By Schur's Lemma the irreducible subrepresentations are all one-dimensional (since $K_\mathbb A^\times$ is abelian). For a character $\Psi$ on $K_\mathbb A^\times$, let $\mathscr{M}_k(\calN,\Psi)$ denote the subspace of $\scrM_k(\calN,\psi,\theta)$ consisting of all $f$ for which $f^{\tilde{s}}=\Psi(\tilde{s})f$ and let $\mathscr{S}_k(\calN,\Psi)\subset \mathscr{M}_k(\calN,\Psi)$ denote the subspace of cusp forms. If $s\in K^\times$ then $f^{s}=f$. It follows that $\mathscr{M}_k(\calN,\Psi)$ is nonempty only when $\Psi$ is a Hecke character. We now have a decomposition $$\scrM_k(\calN,\psi,\theta)=\bigoplus_\Psi \mathscr{M}_k(\calN,\Psi),$$ where the direct sum is taken over Hecke characters $\Psi$ which extend $\psi\psi_\infty$.

If $f=(f_1,...,f_h)\in \scrM_k(\calN,\psi,\theta)$, then each $f_{\lambda}$ has a Fourier expansion
$$f_{\lambda}(\tau)=a_{\lambda}(0)+\sum_{0\ll \xi\in\mathcal{I}_{\lambda}} a_{\lambda}(\xi) e^{2\pi i \mbox{tr} (\xi\tau)}.$$

If $\mathfrak{m}$ is an integral ideal then we define the $\mathfrak{m}$-th `Fourier' coefficient of $f$ by 
\begin{displaymath}
C(\mathfrak{m},f)=\left\{ \begin{array}{ll}
N(\mathfrak{m})^{\frac{k_0}{2}}a_{\lambda}(\xi)\xi^{-\frac{k}{2}-im}& \textrm{if $\mathfrak{m}=\xi\mathcal{I}_{\lambda}^{-1}\subset\calO$}\\
0 & \textrm{otherwise}\\
\end{array} \right.
\end{displaymath}
where $k_0=\mbox{max}\{k_1,...,k_n\}$.

Given $f\in\scrM_k(\calN,\psi,\theta)$ and $y\in G_\mathbb A$ define a slash operator by setting $(f\mid y)(x)=f(xy^{\iota})$. 

For an integral ideal $\mathfrak{r}$ define the shift operator $B_{\mathfrak{r}}$ by $$f\mid B_{\mathfrak{r}}=N(\mathfrak{r})^{-\frac{k_0}{2}} f\mid \left(
  \begin{array}{ c c }
1& 0 \\
      0&\tilde{\mathfrak{r}}^{-1}
  \end{array} \right).$$
  The shift operator maps $\mathscr{M}_k(\calN,\Psi)$ to $\mathscr{M}_k(\mathfrak{r}\calN,\Psi)$. Further, $C(\mathfrak{m},f\mid B_{\mathfrak{r}})=C(\mathfrak{m}\mathfrak{r}^{-1},f)$. It is clear that $f \mid B_{\mathfrak{r}_1}\mid B_{\mathfrak{r}_2}=f\mid B_{\mathfrak{r}_1\mathfrak{r}_2}$.
  
  For an integral ideal $\mathfrak{r}$ the Hecke operator $T_{\mathfrak{r}}=T_{\mathfrak{r}}^{\calN}$ maps  $\mathscr{M}_k(\calN,\Psi)$ to itself regardless of whether or not $(\mathfrak{r},\calN)=1$. This action is given on Fourier coefficients by \begin{equation}\label{equation:heckeformula}C(\mathfrak{m},f\mid T_{\mathfrak{r}})=\sum_{\mathfrak{m}+\mathfrak{r}\subset\mathfrak{a}} \Psi^*(\mathfrak{a})N(\mathfrak{a})^{k_0-1}C(\mathfrak{a}^{-2}\mathfrak{m}\mathfrak{r},f).\end{equation} Note that if $(\mathfrak{a},\mathfrak{r})=1$ then $B_{\mathfrak{a}}T_{\mathfrak{r}}=T_{\mathfrak{r}}B_{\mathfrak{a}}$.

\section{A newform theory for Hilbert Eisenstein series}\label{section:newformtheory}

In this section we give a very brief outline of the newform theory of spaces of Hilbert Eisenstein series. This theory will play a pivotal role in the proof of Theorem \ref{theorem:main}. A detailed treatment is provided in \cite{Linowitz-Atwill}.

Fix a space $\scrM_k(\calN, \Psi)\subset \scrM_k(\calN,\psi)$ where $\psi: (\calO/\calN)^\times\rightarrow\mathbb{C}^\times$ is a numerical character, $\Psi$ is a Hecke character extending $\psi$ and $k\in  \Z^n$. Let $\scrE_k(\calN, \Psi)$ be the orthogonal complement of $\scrS_k(\calN,\Psi)$ in $ \scrM_k(\calN, \Psi)$ with respect to the Petersson inner product (i.e.  $\scrE_k(\calN, \Psi)$ is the subspace of Eisenstein series). It is well known that $ \scrM_k(\calN, \Psi)= \scrS_k(\calN, \Psi)$ unless $k_1=\cdots = k_n\geq 0$. Thus we may abuse notation and make the identification $k=(k,...,k)$ for some integer $k$. We will assume throughout that $k\geq 3$.

We begin with a construction due to Shimura (\cite[Prop. 3.4]{shimura-duke})

\begin{proposition}\label{proposition:existence}
Let $\calN_1, \calN_2$ be integral ideals and $\psi_1$ (respectively $\psi_2$) be a character, not necessarily primitive, on the strict ideal class group modulo $\calN_1$ (respectively $\calN_2$) such that
\begin{align*}
\psi_1(\nu\mathcal{O}) = \mbox{sgn}(\nu)^a\qquad \mbox{for } \nu\equiv 1\pmod{\calN_1}\\ 
 \psi_2(\nu\mathcal{O}) = \mbox{sgn}(\nu)^b\qquad \mbox{for } \nu\equiv 1\pmod{\calN_2},
\end{align*}
where  $a,b\in \Z^n$ and $a+b\equiv k\pmod{2\Z^n}$. Then there exists $E_{\psi_1,\psi_2}\in\scrE_k(\calN_1\calN_2,\Psi)$, where $\Psi$ is the Hecke character such that $\Psi^*(\mathfrak r)=(\psi_1\psi_2)(\mathfrak r)$ for $\mathfrak r$ coprime to $\calN_1\calN_2$,  such that for any integral ideal $\mathfrak{m}$, 

\begin{equation}\label{equation:coeff}
C(\mathfrak{m}, E_{\psi_1,\psi_2})=\sum_{\mathfrak{r}\mid \mathfrak{m}} \psi_1(\mathfrak{m}\mathfrak{r}^{-1})\psi_2(\mathfrak{r}) N(\mathfrak{r})^{k-1}.
\end{equation}
\end{proposition}

It is easy to see that, in analogy with the cuspidal case, if $\frakp$ is a prime not dividing $\calN_1\calN_2$ then $E_{\psi_1\psi_2}$ is an eigenform of the Hecke operator $T_\frakp$ with eigenvalue $C(\frakp,E_{\psi_1,\psi_2})$.

We say that a form $E_{\psi_1,\psi_2}$ is an \textit{Eisenstein newform} of level $\calN$ if the characters $\psi_1$ and $\psi_2$ are both primitive and $\calN=\mathfrak f_{\psi_1}\mathfrak{f}_{\psi_2}$. We denote by $\scrE^{new}_k(\calN,\Psi)$ the subspace of $\scrE_k(\calN,\Psi)$ generated by newforms of level $\calN$.

As in the cuspidal case, $\scrE_k(\calN,\Psi)$ has a basis consisting of shifts of newforms of level $\calM$ dividing $\calN$ (\cite[Prop. 1.5]{wiles}, \cite[Prop. 3.11]{Linowitz-Atwill}).

\begin{proposition}\label{proposition:decomp}
Let notation be as above. We have the following decomposition of $\scrE_k(\calN,\Psi)$:$$\scrE_k(\calN,\Psi)=\bigoplus_{\mathfrak{f}_\Psi\mid \mathcal{M}\mid \calN}\bigoplus_{\qquad\mathfrak{r}\mid \calN\mathcal{M}^{-1}} \scrE^{(new)}_k(\mathcal{M},\Psi)\mid B_{\mathfrak{r}}.$$
\end{proposition}

The eigenvalues of an Eisenstein newform with respect to the Hecke operators $\{ T_\frakp : \frakp\nmid \calN\}$ distinguish it from other newforms. In particular we have the following strong multiplicity-one theorem (Theorem 3.6 of \cite{Linowitz-Atwill}).

 \begin{theorem}\label{theorem:eigenvalues}
Let $\calM,\calN$ be integral ideals and $E_{\psi_1,\psi_2}\in \scrE^{(new)}_k(\calM,\Psi)$ and $E_{\phi_1,\phi_2} \in \scrE_k^{(new)}(\calN,\Psi)$ be newforms such that $$C(\frakp,E_{\psi_1,\psi_2})=C(\frakp,E_{\phi_1,\phi_2})$$ for 
a set of finite primes of $K$ having Dirichlet density strictly greater than $\frac{1}{2}$. Then $\calM=\calN$ and $E_{\psi_1,\psi_2}=E_{\phi_1,\phi_2}$.
\end{theorem}

\section{Proof of theorem}

Fix an integral ideal $\calN$ of $K$, a weight vector $(k,\dots, k)\in \Z^n$ with $k\geq 3$ and consider the space $\scrM_k(\calN)$ of adelic Hilbert modular forms of level $\calN$, weight $k$ and trivial character. We will characterize the Hilbert modular cusp forms of $\scrM_k(\calN)$ in terms of the size of their Fourier coefficients.

\begin{theorem}\label{theorem:main}
Fix a real number $\epsilon\in \left(0, \frac{5k-7}{10}\right]$. If the Fourier coefficients of $f\in\scrM_k(\calN)$ satisfy 
\begin{equation}\label{equation:bound}
|C(\frakm,f)| \ll_{f,\epsilon} N(\frakm)^{k-1-\epsilon},
\end{equation}
then $f$ is a cusp form.
\end{theorem}

\begin{proof}
The space $\scrM_k(\calN)$ decomposes into a direct sum $$\scrM_k(\calN)=\scrS_k(\calN)\oplus\scrE_k(\calN),$$ where $\scrE_k(\calN)$ is the subspace of Eisenstein series. Shahidi \cite{Shahidi} has shown that the Fourier coefficients of cusp forms satisfy $|C(\frakm,f)|\ll_f N(\frakm)^{\frac{k-1}{2}+\frac{1}{5}}$ and hence satisfy (\ref{equation:bound}) as well. Thus it suffices to show that if $f\in\scrE_k(\calN)$ satisfies (\ref{equation:bound}) then $f=0$. We have a decomposition 
\begin{equation}\label{equation:heckedecomp}
\scrE_k(\calN)=\bigoplus_{\Psi} \scrE_k(\calN,\Psi),
\end{equation}

where the sum is taken over Hecke characters extending the trivial character $$\psi_0: \left(\calO/\calN\right)^\times\rightarrow \C^\times.$$

Let $\T^\calN$ denote the algebra generated by all Hecke operators $T_\frakm$ with $\frakm$ an integral ideal coprime to $\calN$.
It is  well-known that for each direct summand in (\ref{equation:heckedecomp}) we can find an operator $T\in\T^\calN$ such that $T$ acts on this summand by multiplication by a non-zero scalar and annihilates all of the other summands.

Observing that if $f\in \scrE_k(\calN)$ satisfies (\ref{equation:bound}) then so too does $f\mid T$ for all $T\in\T^\calN$ (as follows from the action of $T_\frakm$ on Fourier coefficients; see (\ref{equation:heckeformula})), we see that it suffices to fix a direct summand $\scrE_k(\calN,\Psi)$ of (\ref{equation:heckedecomp}) and prove that if $f\in\scrE_k(\calN,\Psi)$ and $f$ satisfies (\ref{equation:bound}) then $f=0$.

The newform theory of Hilbert Eisenstein series \cite{Linowitz-Atwill} shows that we have
\begin{equation}\scrE_k(\calN,\Psi)=\bigoplus_{\substack{(\psi_1,\psi_2)\\ \psi_1\psi_2=\Psi^*}} \bigoplus_{\fraka\mid \calN\frakf_{\psi_1}^{-1}\frakf_{\psi_2}^{-1} } \C E_{\psi_1,\psi_2}\mid B_\fraka, \end{equation}
where the first sum is over pairs $(\psi_1,\psi_2)$ of primitive characters on the strict class group of $K$ having the property that $\Psi^*(\mathfrak r)=\psi_1(\mathfrak r)\psi_2(\mathfrak r)$ for all integral ideals $\mathfrak r$ coprime to $\frakf_{\psi_1}\frakf_{\psi_2}$. 
	
We therefore fix an Eisenstein newform $E_{\psi_1,\psi_2}$ of level $\calM$ dividing $\calN$ which satisfies (\ref{equation:bound}) and write \begin{equation}\label{equation:fdecomp}f=\sum_{\frak a\mid\calN\calM^{-1}} \lambda_\fraka E_{\psi_1,\psi_2}\mid B_\fraka,\end{equation} where the coefficients $\lambda_\fraka$ are all complex numbers. 

We show that each of the coefficients $\lambda_\fraka$ is equal to zero. Our argument is by induction. For each $r\geq 0$ we show that if a divisor $\fraka$ of $\calN\calM^{-1}$ has $r$ prime factors (with multiplicity) then $\lambda_\fraka=0$.  

Suppose first that $r=0$ so that $\fraka=(1)$. Let $\frakP$ be a prime ideal which represents the trivial element of the strict ideal class group modulo $\calN$. It is clear that $\frakP\nmid \calN$. Then $\psi_1(\frakP)=\psi_2(\frakP)=1$. Note as well that if $\frakb$ is a divisor of $\calN\calM^{-1}$ with at least one prime divisor then $$C(\frakP, E_{\psi_1,\psi_2}\mid B_{\frakb})=C(\frakP\frakb^{-1},E_{\psi_1,\psi_2})=0$$ since $\frakP\frakb^{-1}$ is not integral. Computing the $\frakP$-th coefficients in (\ref{equation:fdecomp}) now yields:
$$|\lambda_{(1)}|(1+N(\frakP)^{k-1})\ll N(\frakP)^{k-1-\epsilon}.$$ The Chebotarev Density theorem implies that there are infinitely many prime ideals $\frakP$ which represent the trivial strict ideal class modulo $\calN$, hence as $N(\frakP)\rightarrow \infty$ we obtain a contradiction unless $\lambda_{(1)}=0$.

We now suppose that $r\geq 1$ and that we have shown that $\lambda_{\frakc}=0$ for any divisor $\frakc$ of $\calN\calM^{-1}$ having fewer than $r$ prime factors. Let $\fraka$ be a divisor of $\calN\calM^{-1}$ having exactly $r$ prime factors (not necessarily distinct) and write $\fraka=\frakP_1\dots\frakP_r$ as a product of primes. Let $\frakP$ be a prime ideal which represents the trivial element of the strict ideal class group modulo $\calN$ and set $\frakm=\fraka\frakP=\frakP_1\dots\frakP_r\frakP$. As above, if $\frakb\neq \fraka$ is a divisor of $\calN\calM^{-1}$ with at least $r$ prime factors then we have $$C(\frakm, E_{\psi_1,\psi_2}\mid B_{\frakb})=C(\frakm\frakb^{-1},E_{\psi_1,\psi_2})=0.$$ Computing the $\frakm$-th coefficients in (\ref{equation:fdecomp}) now yields:

$$|\lambda_\fraka|(1+N(\frakP)^{k-1})\ll N(\frakP)^{k-1-\epsilon}.$$ As $\frakP$ was selected from an infinite set of primes (a set which in fact has positive Dirichlet density), by letting $N(\frakP)\rightarrow \infty$ we obtain a contradiction unless $\lambda_\fraka=0$.

\end{proof}

% ------------------------------------------- Begin References
% -------------------------------------------

\end{document}